\documentclass[
final
, nomarks
]{dmtcs-episciences}

\usepackage[utf8]{inputenc}
\usepackage{subfigure}
\usepackage[english]{babel}
\usepackage{amssymb}
\usepackage{amsthm}
\usepackage{amsmath}
\usepackage{a4wide}
\usepackage{esvect}
\usepackage{hyperref}
\usepackage{verbatim}
\usepackage{color}
\usepackage{geometry}
\usepackage{framed}
\usepackage[colorinlistoftodos,prependcaption,textsize=tiny]{todonotes}

\newtheorem{thm}{Theorem}

\newtheorem{clm}[thm]{Claim}

\newcommand\cF{{\mathcal F}}
\newcommand\cG{{\mathcal G}}

\newcommand\cL{{\mathcal L}}

\newcommand\cP{{\mathcal P}}

\newcommand\cQ{{\mathcal Q}}

\newcommand{\ignore}[1]{}

\author{
D\'aniel Gerbner
\and
M\'at\'e Vizer
}

\title{On non-adaptive majority problems of large query size}

\affiliation{Alfr\'ed R\'enyi Institute of Mathematics} 


\received{2021-01-13}

\revised{2021-08-06} 

\accepted{2021-10-27}

\publicationdetails{23}{2021}{3}{15}{7084}
\begin{document}
\maketitle

\begin{abstract}

We are given $n$ balls and an unknown coloring of them with two colors. Our goal is to find a ball that belongs to the larger color class, or show that the color classes have the same size. We can ask sets of $k$ balls as queries, and the problem has different variants, according to what the answers to the queries can be. These questions has attracted several researchers, but the focus of most research was the adaptive version, where queries are decided sequentially, after learning the answer to the previous query. Here we study the non-adaptive version, where all the queries have to be asked at the same time.


\end{abstract}

\section{Introduction}

\indent

A widely studied problem in combinatorial search theory is the so-called \textit{Majority Problem}. We are given $n$ indexed balls - say the set $[n]=\{1,2,...,n\}$ - as an input, each colored in some way unknown to us with one of two colors. A ball $i \in [n]$ is called \textit{majority ball} if there are more than $\frac{n}{2}$ balls in the input set that have the same color as $i$. We would like to find a ball of the majority color or show that there is no majority color by asking subsets of $[n]$, that we call \textit{queries}. We would like to determine the minimum number of queries needed in the \textbf{worst case} with an optimal strategy if all the queries are fixed at the beginning. We call this the \textit{non-adaptive} version of the Majority Problem. 
If the queries may depend on the answers to the previous ones, that we call the \textit{adaptive} version of the Majority Problem.

We still need to describe what kind of queries one can use. Each query corresponds to a subset of size $k$ of the input set. There are different variants of this problem, according to what the answer to a query is.
More precisely, a \textit{model} of the Majority Problem is given by $[n]$, the number of colors, the size of the queries $k$, the possible answers and whether it is adaptive or non-adaptive. In this paper we deal only with two colors, in the non-adaptive case. Sometimes we look at the set of queries as a hypergraph $\mathcal{Q}$, with the queries being the (hyper)edges. We refer to balls also as vertices. 

\textbf{Models}

The most basic model is the \textit{pairing model}. In this model the size of a query is two, and the answer 
is YES, if the two balls have the same color and NO otherwise.

We note that the adaptive version of this problem, when the number of colors is not limited was investigated by Fisher and Salzberg \cite{FS1982}, who proved that $\lceil 3n/2 \rceil - 2$ queries are necessary and sufficient. If the number of colors is two, then Saks and Werman \cite{SW1991} proved that the minimum number of queries needed is $n-b(n)$, where $b(n)$ is the number of $1$'s in the diadic form of $n$. The non-adaptive version was studied in \cite{GKPP2013}.

In this paper we deal with generalizations of the pairing model, when we ask queries of larger size. The first model of this kind was introduced and investigated by De Marco, Kranakis and Wiener \cite{DKW2011}, then many related results appeared in the literature \cite{B2014,cgp,DK2015,EH2015,GKPPVW2015,glv}. However, most of them studied only the adaptive case.



The authors considered some of these models in \cite{GV2017}, and improved the existing bounds in the adaptive case. Here we investigate the non-adaptive versions. We remark that the first arXiv version of \cite{GV2017} contained a section on the non-adaptive case, thus most of our results. 
Following the suggestion of an anonymous referee, we removed that part from \cite{GV2017}, with the plan of publishing it separately. We also extended the results slightly.

\smallskip

\textbf{Hypergraph language}.

The queries can be considered as edges of a hypergraph, which we call \textit{query hypergraph} and usually denote by $\cQ$. 
We introduce some hypergraph properties, that we will use later.
A hypergraph has \textit{Property B} if its vertices can be colored with two colors such that there is no \textit{monochromatic} edge in the hypergraph, i.e.~ an edge with vertices of the same color. For $k\ge 1$, let us denote by $\mathsf{m}(k)$ the cardinality of the edge set of a smallest $k$-uniform hypergraph that does not have Property B. This parameter is widely studied, the best lower bound on $\mathsf{m}(k)$ we are aware of is $\Omega(2^k\sqrt{k/\log{k}})$ due to Radhakrishnan and Srinivasan \cite{RS2000}. For $n\ge 2k-1$ we also consider $\mathsf{m}(k,n)$, which is the cardinality of a smallest $k$-uniform hypergraph with $n$ vertices that does not have Property B. Obviously if $n$ is large enough, then we have $\mathsf{m}(k,n)=\mathsf{m}(k)$.

We also use a similar notion that we call \textit{Property $C$}. Let $k\ge 2$. We call a coloring of a $k$-set with two colors \textit{balanced} if the cardinality of the two color classes differ by at most one. We say that a hypergraph has Property C if its vertices can be colored with two colors such that every edge is balanced. Let us denote by $\mathsf{d}(k)$ the cardinality of the edge set of a smallest $k$-uniform hypergraph that does not have Property C. We also consider $\mathsf{d}(k,n)$, which is the cardinality of the edge set of a smallest $k$-uniform hypergraph with $n$ vertices that does not have Property C. Obviously, if $n$ is large enough, then we have $\mathsf{d}(k,n)=\mathsf{d}(k)$.


For odd $k$,
Eppstein and Hirschberg \cite{EH2015} proved $\mathsf{d}(k)\le k+3\log k+4$. If $k$ is even, this problem can also be formulated as we are looking for the smallest hypergraph with a positive discrepancy. If $k=2 \mod 4$, then it is easy to see that $\mathsf{d}(k)=3$.
Let $snd(i)$ be the smallest positive integer that does not divide $i$. Alon, Kleitman, Pomerance, Saks and Seymour \cite{AKPSS1987} and Cherkashin and Petrov \cite{CP2017} proved that there exist constants $c_1$ and $c_2$ such that if $k=0 \mod 4$, then
$$c_1\frac{\log snd(k/2)}{\log\log snd(k/2)}\le \mathsf{d}(k) \le c_2\log snd(k). $$

\textbf{Structure of the paper}.

The rest of the paper is organised as follows: in Section 2 we introduce the models and state the known results in the adaptive setting. In Section 3, 4 and 5 we state and prove our results regarding the different models.

\section{Models}

In this section we define the models we study, and state the results known in the adaptive case. Three of the models were introduced by De Marco and Kranakis \cite{DK2015}, and their (adaptive) bounds were improved by Eppstein and Hirschberg \cite{EH2015} and by the authors \cite{GV2017}. The last one was introduced by Borzyszkowski \cite{B2014}, who found the exact answer in the adaptive case. In each of these models, we are given $n$ indexed balls, colored with two colors, and we ask queries of size $k$ ($k \ge 2$) that we denote by $Q$. The models differ only in the possible answers to the queries, thus we emphasize what the answers can be. We give each model an abbreviation that we indicate after its name. If the abbreviation is $X$, then we will denote by $A(X,k,n)$ and $N(X,k,n)$ the number of queries needed to ask in the worst case of the adaptive/nonadaptive version of that model, respectively.



\subsection{Models, adaptive results}

$\bullet$ \texttt{Output (or Partition) Model = OM}:

\begin{framed}
\noindent$\mathbf{Answer:}$ $\{Q',Q''\}$, a partition of $Q$, where $Q'$ is the set of balls of one color and $Q''$ of the other color.
No indication is provided about which of the colors is in $Q'$.
\end{framed}

In the following theorem, the upper bound if $n$ is even is due to De Marco and Kranakis \cite{DK2015}, while the other bounds are due to Eppstein and Hirschberg \cite{EH2015}.

\begin{thm}[Eppstein, Hirschberg \cite{EH2015}, De Marco, Kranakis \cite{DK2015}] For all $2 \le k \le n$ we have

\begin{displaymath}
A(OM,k,n)\ge
\left\{ \begin{array}{l l}
\lceil \frac{n-1}{k-1} \rceil & \textrm{if $k$ is odd,}\\
\frac{n}{k-1} - O(n^{1/3}) & \textrm{if $k$ is even,}\\
\end{array}
\right.
\end{displaymath}

\begin{displaymath}
A(OM,k,n)\le
\left\{ \begin{array}{l l}
\lceil \frac{n-1}{k-1} \rceil & \textrm{if $n$ is even,}\\
\lceil \frac{n-2}{k-1} \rceil & \textrm{if $n$ is odd.}\\
\end{array}
\right.
\end{displaymath}

\end{thm}

$\bullet$ \texttt{Counting Model = CM}:

\begin{framed}

\noindent$\mathbf{Answer:}$ a number $i \le k/2$ such that the query has exactly $i$ balls of one of the color classes (thus $k-i$ balls of the other color class).
\end{framed}

In the following result, the upper bound is due to Eppstein and Hirschberg \cite{EH2015} , while the lower bound is due to Gerbner and Vizer \cite{GV2017}.

\begin{thm}[Eppstein, Hirschberg \cite{EH2015},Gerbner, Vizer \cite{GV2017}] For all $2 \le k < n$ we have
$$ \frac{6n}{5k+6}-c(k)\le A(CM,k,n) \le \frac{n}{\lfloor \frac{k}{2} \rfloor}  + O(k),$$
where $c(k)$ depends only on $k$.
\end{thm}

$\bullet$ \texttt{General (or Yes-No) Model = GM}:

\begin{framed}

\noindent$\mathbf{Answer:}$ YES, if there exist two balls of different colors, NO otherwise.

\end{framed}

\begin{thm}[Gerbner, Vizer \cite{GV2017}]\label{ayesno} For any $2 \le k < n$ with $2k-1\le n$ we have $$\frac{3n+5}{4} \le A(GM,k,n)\le n-k+\mathsf{m}(k).$$

\end{thm}

$\bullet$ \texttt{Borzyszkowski's Model = BM}, \cite{DKW2011,B2014}:


\begin{framed}

\noindent$\mathbf{Answer:}$ YES, if there exist two balls of different colors, and such a pair is pointed out, NO if all balls have the same color.

\end{framed}

\begin{thm}[Borzyszkowski \cite{B2014}]\label{borz} For all $3 \le k \le n$ with $2k-3 \le n$ we have
\begin{displaymath}
A(BM,k,n)=
\left\{ \begin{array}{l l}
\frac{n}{2}+k-2 & \textrm{if $n$ is even,}\\
\lfloor \frac{n}{2} \rfloor + k-3 & \textrm{if $n$ is odd.}\\
\end{array}
\right.
\end{displaymath}

\end{thm}

\subsection*{Basic inequalities}

By definition, the following inequalities hold for the above models
for all $2 \le k \le n$:
$N(OM,k,n) \le N(CM,k,n) \le N(GM,k,n)$ and $N(OM,k,n) \le N(BM,k,n) \le N(GM,k,n).$

\section{Output model}

\begin{thm}\label{nom} For $2 \le k \le n$ we have

\begin{displaymath}
N(OM,k,n)=
\left\{ \begin{array}{l l}
\lceil \frac{n-1}{k-1} \rceil & \textrm{if $n$ is even,}\\
\lceil \frac{n-2}{k-1} \rceil & \textrm{if $n$ is odd.}\\
\end{array}
\right.
\end{displaymath}

\end{thm}

\begin{proof}
A hypergraph is connected if we cannot partition the underlying set into two parts such that no edge contains vertices from both parts. Observe that the least number of edges of a connected $k$-uniform hypergraph on an $n$-element underlying set is $\lceil\frac{n-1}{k-1}\rceil$. Indeed, we can build one by taking an arbitrary set of size $k$, and add new edges that intersect the union of the earlier ones in one element. On the other hand, if we are given a connected hypergraph, and we take an arbitrary set from it, there must exist another set that intersects it, then another set that intersects the union of the earlier ones, and so on. Similarly, if a hypergraph has $d$ connected components, then the least number of edges it can have is $\lceil\frac{n-d}{k-1}\rceil$.

It is easy to see that if the query hypergraph $\cQ$ is connected, we find more than a majority ball; we find the partition to color classes. This proves the upper bound in case $n$ is even. If $n$ is odd, we can ask a connected query hypergraph on $n-1$ vertices. A ball that is a majority ball among those vertices is also a majority ball in the whole set of balls, while if there is no majority ball there, then the remaining ball is a majority ball. This finishes the proof of the upper bound. 

Let us continue with the lower bound. If the query hypergraph $\cQ$ is disconnected and $n$ is even, we consider an arbitrary partition of the underlying set of balls into two parts with no edge intersecting both. Then the answers might be that in both parts one of the color classes is larger by the same number $l>0$. What we mean is that we take a coloring with the above property, and then the answers do not tell any further information. If the larger color class is the same color in the parts, there is majority, otherwise there is not. We cannot find out which one is the case, hence we cannot show a majority ball. This finishes the proof in case $n$ is even. If $n$ is odd, and the underlying set can be partitioned into three parts such that no edge contains vertices from two parts, then it is possible that in each part the difference between the size of the color classes is 1 or 2. In this case no majority ball can be shown. We have already mentioned that is easy to see that at least $\lceil\frac{n-2}{k-1}\rceil$ queries are needed to avoid this.
\end{proof}

\section{Counting model}

\begin{thm}\label{ncmu} For $2 \le k, n$ with $2k-1 \le n$ we have 

\begin{displaymath}
N(CM,k,n) \le
\left\{ \begin{array}{l l}
 n-k+1 & \textrm{if $k$ is even},\\
(n-k+1)\times(1+\mathsf{d}(k-1,n)) & \textrm{if $k$ is odd.}\\
\end{array}
\right.
\end{displaymath}

\end{thm}

\begin{proof}
If $k$ is even, then we ask all the $k$-sets containing a given $(k-1)$-set $A \subset [n]$. Note that - using that $k$ is even - we know that $i,j \in [n]\setminus A$ have the same color if and only if the answers to the queries $A \cup \{i\}$ and $A \cup \{j\}$ are the same. Thus we can partition $[n] \setminus A$ into two parts such that the balls in each part have the same color. Additionally, if we ever get different answers for two such queries, then we know the number of balls of the corresponding colors inside $A$ and we can choose a majority ball or find out that there is none. If this is not the case, then - knowing that $n \ge 2k-1$ - we have that all balls in $[n] \setminus A$ are of the majority color. 

If $k$ is odd, then the previous argument does not work, as the answers to some queries could be that there are $(k+1)/2$ and $(k-1)/2$ balls of the two colors, even if the balls added to $A$ are of different color. However, it cannot happen if $A$ contains different number of red and blue balls. So we use unbalanced colorings here. Let us take a $(k-1)$-uniform hypergraph $\mathcal{F}$ on $[n]$ that does not have Property C and has cardinality $\mathsf{d}(k-1,n)$. Moreover, if every pair of edges in $\mathcal{F}$ has intersection of size at least $(k-3)/2$, then we add another edge of size $k-1$ that intersects one of them in a set of size less than $(k-3)/2$. This way we get a hypergraph $\mathcal{F}'$, and the query set consists of all $k$-sets containing edges of $\mathcal{F}'$. There is a set $F\in\mathcal{F}'$ that is unbalanced, i.e. it contains at least $(k+1)/2$ ball of the same color, say blue. 

If every answer (to every query) were the number $(k-1)/2$, then all the balls not in $F$ would be red. In this case there would be exactly $(k+1)/2$ blue balls and every member of $\mathcal{F}'$ would contain at least $(k-1)/2$ of them. Then their intersection would have size at least $(k-3)/2$, a contradiction. That is why we added the additional set to $\mathcal{F}$.

Thus there is an answer different from $(k-1)/2$. Now similarly to the case when $k$ is even, we can find $F$, and then we know the relation of the outside balls to each other, and the number of the balls of the corresponding colors inside $F$. 
\end{proof}

\begin{thm}\label{ncml}

For $2 \le k$ and sufficiently large $n$ we have

\begin{displaymath}
N(CM,k,n) \ge
\left\{ \begin{array}{l l}
2n/(k+1) & \textrm{if $k$ is even and $n$ is even},\\
(2n-1)/(k+1) & \textrm{if $k$ is even and $n$ is odd}.\\
\lceil \frac{n}{k}\mathsf{d}(k-1,n-1) \rceil & \textrm{if $k$ is odd, $n$ is even and large enough},\\
\lceil \frac{n-1}{2k} \mathsf{d}(k-1,n-1) \rceil & \textrm{if $k$ is odd, $n$ is odd and large enough}.\\
\end{array}
\right.
\end{displaymath}

\end{thm}

\begin{proof}
To prove the lower bound for $k$ even and $n$ even, 
we first show that any query can contain at most one vertex of degree one. Indeed, suppose a query $Q$ contains $i$ and $j$ (with $i \neq j$) of degree one, then it is possible that the answer to $Q$ is the number $(k-2)/2$ and there are $(k-2)/2$ red and $(k-2)/2$ blue balls besides $i$ and $j$ in $Q$. Furthermore, it is possible that altogether there are $n/2$ blue and $(n-4)/2$ red balls besides $i$ and $j$. We know that $i$ and $j$ have the same color, but we do not know if it is blue or red. Thus we do not know if there is a majority color or not. 

Let $q$ denote the number of queries, and $x\le q$ be the number of vertices of degree one. First we show that there is no vertex of degree 0. Indeed suppose that there is a vertex $v$ of degree 0. Then it is possible that among the other vertices, the number of blue balls is larger than the number of red balls by one. In this case we do not know that there is a majority ball or not, as we do not know anything about the color of $v$. Therefore, we have at least $n-x$ vertices of degree at least 2. Thus, the number of pairs $(Q,w)$ where $Q$ is a query, $w$ is a vertex and $w\in Q$ (i.e. the sum of the degrees) is at least $2(n-x)+x$. On the other hand, this number is exactly $kq$. Thus we have $kq\ge 2n-x\ge 2n-q$ and rearranging gives the bound. 

In the other cases below, we just state and prove the degree conditions that are needed to obtain the desired bound, and omit the similar easy calculation that finishes the proof.

In the case $k$ is even and $n$ is odd, it is enough to show that at most one query can contain two elements of degree one. We prove it by contradiction, since otherwise we can get two monochromatic pairs, and it is possible that there are $(n-1)/2$ blue and $(n-7)/2$ red balls besides those four balls. In this case we cannot show a majority ball.

\vspace{1mm}

In the case $k$ is odd, while $n$ is even and large enough, it is enough to show that for every $i \in [n]$, its degree is at least $\mathsf{d}(k-1,n-1)$. This is true, since otherwise we can color the hypergraph with vertex set $\cup \{Q : i \in Q \} \setminus \{i\}$ and edge set $\{Q \setminus \{i\} : i \in Q\}$ (the open neighborhood hypergraph or link hypergraph of $i$) in a balanced way, thus we do not get any information about the color of $i$. Then - using that $n$ is large enough and even - we can color the remaining elements (i.e. $[n] \setminus \cup \{Q : i \in Q \}$) such that the coloring of all the balls is balanced. But then it depends on the color of $i$ if there is a majority color or not. 

If $n$ is odd, a similar argument shows that all but one of the balls have degree at least $\mathsf{d}(k-1,n-1)/2$. Indeed, otherwise there are $i,j \in [n]$ with $i\neq j$ such that less than $\mathsf{d}(k-1,n-1)$ queries contain at least one of $i$ or $j$. Let $\cF$ be the hypergraph that has these queries as edges. Let us remove $i$ and $j$ from them and for those queries containing both $i$ and $j$, we add a new ball $s \not\in [n]$ instead. The resulting $(k-1)$-uniform hypergraph $\cF'$ has less than $\mathsf{d}(k-1,n-1)$ edges, thus it has Property C. This gives a coloring that is balanced on every edge of $\mathcal{F}'$. Let red be the color of $s$. This coloring can be extended to all the balls except for $i$ and $j$ such a way that there are $(n-1)/2$ blue and $(n-3)/2$ red balls among the balls in $[n]\setminus \{i,j\}$. Then the answers to queries containing neither $i$ nor $j$ are according to this coloring. Moreover, a query $Q$ containing exactly one of them can also be answered according to this coloring without knowing the colors of $i$ and $j$, as $Q\setminus \{i,j\}$ is balanced. Finally the answer to queries containing both $i$ and $j$ is the number $(k-1)/2$. It is easy to see that if at least one of $i$ and $j$ is red, the answers are consistent with the coloring. Thus any color can be minority (i.e. not majority), hence a ball different from $i$ and $j$ cannot be the majority ball, as we know its color. But $i$ can be red and $j$ blue, or the other way around. In that case the red ball is in minority, thus we cannot say that $i$ or $j$ is a majority ball, finishing the proof.
\end{proof}

The above theorem can be improved with similar, but more involved arguments, as we show below. For simplicity, we only deal with the case when $k$ and $n$ are both even.

\begin{thm}\label{new} If $k$ is even and large enough, and $n$ is even, then we have $$N(CM,k,n)\ge \frac{n(5ik-k+i+i^2)-2(k-i)}{(2k+3+i)ik}$$ for every $k/2<i<k$. 
In particular, for every $\varepsilon$, there is a $k_0$ such that if $k>k_0$ is even, then there is an $n_0(k)$, such that if $n>n_0(k)$ is even, then $N(CM,k,n)>(11/5-\varepsilon)n/k$.
\end{thm}

We remark that for any specific $k$, one can easily obtain from the first part of the above theorem a lower bound for $N(CM,k,n)$ with a simple calculation. 
For example, if $k=8$, then $i=5$ gives the best lower bound, and it is slightly larger than the lower bound $2n/9$ from Theorem \ref{ncml}.

\begin{proof}
Let $\cF$ be the subhypergraph of the query hypergraph $\cQ$ having as edges those queries that contain at least $i+1$ vertices of degree at most 2. 
Let $\cF'$ be the multi-hypergraph obtained from $\cF$ by deleting the vertices of degree more than 2 from each edge. Note that $\cF'$ may be non-uniform.

\begin{clm} The total size of $\cF'$, i.e. the sum of the edge sizes is at most $2k+(i+1)n/i$.

\end{clm}

\begin{proof}[Proof of Claim] First we show that $\cF'$ is a linear hypergraph. Indeed, otherwise there are two queries $Q_1\in \cQ$ and $Q_2\in\cQ$ sharing two balls $x$ and $y$ that do not appear in any other query. Then one can color the other balls of $Q_1$ and $Q_2$ such a way that both $Q_1$ and $Q_2$ have $(k-2)/2$ red and $(k-2)/2$ blue balls besides $x$ and $y$. Furthermore, one can color the other balls such that there are $n/2$ blue and $(n-4)/2$ red balls besides $x$ and $y$. Then the answer to every query is according to this coloring, and the answer to $Q_1$ and $Q_2$ is the number $(k-2)/2$. Then we know $x$ and $y$ have the same color, but we do not know if it is blue or red, thus we do not know if there is a majority color or not.

Next we show that $\cF'$ does not contain any linear cycle covering at most $n/2+3$ vertices. A linear cycle of length $\ell$ consists of $\ell$ edges $h_1,\dots, h_\ell$ such that the only intersections among those are the singleton intersections $v_j$ of $h_j$ and $h_{j+1}$  for every $j$, modulo $\ell$ (i.e. we also have the singleton intersection $v_\ell$ of $h_\ell$ and $h_1$). Let $V=\{v_1,\dots,v_\ell\}$.

Assume that there is a linear cycle of length $\ell$ covering $m\le n/2+3$ vertices, and for every $j$, let $Q_j$ denote the query that $h_j$ was obtained from. 
We define a coloring of all the balls not in $V$ such that this coloring already determines the answers to every query in $Q$, but does not determine whether there is a majority ball (hence it leads to a contradiction).

We first color the vertices that are in some $Q_j$ but not in any $h_{j'}$ with $j\neq j'$, to red. Then we go through the edges $h_j$ in an arbitrary order, and color the vertices in $h_j\setminus V$, such a way that every $Q_j$ contains $(k-2)/2$ blue balls and $(k-2)/2$ red balls not in $V$. This is doable, as only the vertices of $Q_j\setminus h_j$, thus at most $k-(i+1)\le (k-2)/2$ balls of $h_i$ had been colored (to red) when we arrived to $h_i$. Observe that we colored to blue only the vertices covered by the linear cycle, except the vertices $v_1,\dots,v_\ell$ and $\ell\ge 3$. Therefore, we have colored $m-\ell\le n/2$ vertices to blue so far. We claim that so far the number of blue balls is at least the number of red balls. Indeed, for each $j\le \ell$, we colored $(k-2)/2$ vertices in $h_j\setminus V$ to blue. As these sets are vertex-disjoint, there are $\ell(k-2)/2$ blue balls. Similarly, for each $j\le \ell$, we colored $(k-2)/2$ vertices in $Q_j$ to red, thus there are at most $\ell(k-2)/2$ red balls.

Now we can color the remaining balls (besides those in $V$) so that there are $n/2$ blue and $n/2-\ell$ red balls. Then we can give an answer to every $Q_j$ as the number $(k-2)/2$, and answer every other query according to this coloring. Then we know that every ball in $V$ has the same color, but we do not know if it is blue or red, thus we do not know if there is a majority color or not.

This implies that every linear cycle in $\cF'$ has more than $n/2+3$ vertices. Next we show that $\cF'$ contains at most three linear cycles.
Let $\cL$ and $\cL'$ be two linear cycles in $\cF'$ with edges $h_1,\dots,h_\ell$ and $h'_1,\dots,h'_{\ell'}$
Then by the pigeonhole principle they share at least two vertices.
We claim that $\cL$ and $\cL'$ have to share a linear subpath and nothing more (a linear path is a linear cycle without one of the edges, or a single edge, or a single vertex). Indeed, otherwise they share two subpaths, thus following $\cL'$, one leaves $\cL$ and returns to it at least twice. It is easy to see that this way we can find two vertex disjoint linear cycles, a contradiction.

Therefore, for some vertices $x$ and $y$ there are three linear paths $\cP_1,\cP_2,\cP_3$ between $x$ and $y$ that share only these two vertices. These paths define three linear cycles. We claim that these are all the linear cycles in $\cF'$.

A fourth cycle intersects each of our three cycles multiple times, in particular it intersects two of the paths, say $\cP_1$ and $\cP_2$. It has a linear subpath $\cP_4$ between vertices $u$ in $\cP_1$ and $v$ in $\cP_2$ such that $\cP_4$ shares only $u$ and $v$ with our three cycles. Let $l_j$ denote the number of vertices in $\cP_j$, then $l_4+(l_1+l_2)/2>n/2+3$, because $\cP_4$ forms a cycle with any of the two paths between $u$ and $v$ in $\cP_1\cup\cP_2$. On the other hand, $l_1+l_3>n/2+3$ and $l_2+l_3>n/2+3$, thus $(l_1+l_2)/2+l_3>n/2+3$. Adding up the above inequalities, we obtain $l_1+l_2+l_3+l_4>n+6$, a contradiction.

Now we can delete two edges of $\cF'$ to obtain a linear hypergraph $\cF''$ without any linear cycles. We show that this cycle-free property implies that sum of the degrees of $\cF''$ is at most $(i+1)n/i$ (which completes the proof of the claim). Indeed, we can build $\cF''$ the following way. We start with an arbitrary edge, and build a connected component by adding a new edge sharing one vertex with the component. If a component is finished, we start again with a disjoint edge and repeat this. This way every new edge adds $j+1$ to the sum of the degrees for some $j\ge i$, but also increases by $j$ the number of the vertices used. As $(j+1)/j\le(i+1)/i$, we will get the largest total size if we add an edge of size $i+1$ all the time, that is, at most $n/i$ times.
\end{proof}
We also use that any query in $\cQ$ has at most one vertex of degree one, which we obtained at the beginning of the proof of Theorem \ref{ncml}.

Let us count the total sum $D$ of the degrees in $\cQ$. On the one hand, $D$ is obviously $k|\cQ|$, because every query adds $k$ to the sum. But look at this more precisely. 
 Let $D'$ be total sum of the degrees of balls with degree 1 or 2, and $D''$ be total sum of the degrees of balls with degree greater than 2, thus we have $k|\cQ|=D=D'+D''$. Let us look at $D'+D''$ in a different way. For a query $Q$, let us denote by $q'(Q)$ the number balls in $Q$ with degree 1 or 2, and by $q''(Q)$ the number of balls in $Q$ with degree greater than 2. We say that $Q$ adds $q'(Q)$ to $D'$ and $q''(Q)$ to $D''$, since $D'=\sum_{Q\in\cQ}q'(Q)$ and $D'=\sum_{Q\in\cQ}q''(Q)$. 
 
If a query is not in $\cF$, it adds to $D''$ at least $(k-i)/i$ times as much as it adds to $D'$. The queries in $\cF$ altogether add at most $2k+(i+1)n/i$ to $D'$.
These imply that $D'\le iD''/(k-i)+2k+(i+1)n/i=i(k|Q|-D')/(k-i)+2k+(i+1)n/i$, thus $D'\le i|\cQ|+2(k-i)+(i+1)(k-i)n/ik$.

Let $a$ be the number of balls of degree 1, $b$ be the number of balls of degree 2 and $c$ be the number of balls of degree greater than 2. Then $a\le |\cQ|$, $n=a+b+c$, $D'=a+2b$ and $D''\ge 3c$. Thus $n\le |\cQ|/2+D'/2+D''/3$, i.e. $6n\le 3|\cQ|+3D'+2D''=(2k+3)|\cQ|+D'\le (2k+i+3)|\cQ|+2(k-i)+(i+1)(k-i)n/ik$. This implies $|\cQ|\ge \frac{n(5ik-k+i+i^2)-2(k-i)}{(2k+3+i)ik}$. By choosing $i$ to be $k/2+1$, it is easy to see that this lower bound divided by $n/k$ converges to $11/5$, which finishes the proof.
\end{proof}

\section{General model and Borzyszkowski's model}

\begin{thm}\label{nyesnou} For $2 \le k$ and $2k-1 \le n$, we have $$N(BM,k,n) \le N(GM,k,n)\le (n-k+1)\times\mathsf{m}(k-1,n-1).$$ 

\end{thm}

\begin{proof}

We consider a $(k-1)$-uniform hypergraph $\mathcal{F}$ of size $\mathsf{m}(k-1,n-1)$ on $n-1$ vertices that does not have property B. Then we add an additional $n$th vertex. We query each set of size $k$ that contains an edge of $\mathcal{F}$, thus at most $(n-k+1)\mathsf{m}(k-1,n-1)$ sets. If we get a NO answer, we can identify a monochromatic member $F$ of $\mathcal{F}$, which is queried together with every other ball, thus we find out if the other balls have the same color as balls in $F$ or not, i.e. we can identify the color classes. Otherwise, as we know that there is a monochromatic member $F$ of $\mathcal{F}$, every ball in $[n] \setminus F$ is the other color, thus majority ball, because of $2k-1 \le n$. In particular the ball not in members of $\cF$ is a majority ball. 
\end{proof}

\begin{thm}\label{nyesnol}

Let $k \ge 3$ and $n$ be large enough.

$\bullet$ If $n$ is even, then we have 
$f(n,k)\times\mathsf{m}(k-1) \le N(BM,k,n) \le N(GM,k,n)$.

$\bullet$ If $n$ is odd, then we have
$f(n,k)\times\mathsf{m}(k-1) \le N(GM,k,n)$ and

\hspace{4.5cm} $f(n,k)\times \mathsf{m}(k-2) \le N(BM,k,n)$,  where 

\begin{displaymath}
f(n,k):=
\left\{ \begin{array}{l l}
\frac{n}{k} & \textrm{if $n$ is even},\\
\frac{n-1}{2k} & \textrm{if $n$ is odd.}\\
\end{array}
\right.
\end{displaymath}

\end{thm}

\begin{proof}
Let us start with the General Model for even $n$. We prove that every ball $i$ has degree at least $\mathsf{m}(k-1)$. Indeed, otherwise it is possible that the set $\bigcup \{Q : i \in Q\} \setminus \{i\}$ is colored by blue and red in such a way, that no edge in the open neighborhood hypergraph of $i$ is monochromatic. Thus one can answer such a way that the color of $i$ does not change any answers. As $n$ is large enough, these less than $\mathsf{m}(k-1)$ queries contain less than $n/2$ balls, hence we can color the remaining balls such a way that $n/2$ of them are red and $n/2-1$ of them are blue. As $i$ can be of any color, we do not know if there is a majority ball or not.

If $n$ is odd, we will use a similar argument to the proof of Theorem \ref{ncml} to show that all but one balls have degree at least $\mathsf{m}(k-1)/2$. Indeed, otherwise there are less than $\mathsf{m}(k-1)$ queries containing either $i$ or $j$. Let $\cF\subset \cQ$ be the hypergraph having those queries as edges. Let us remove $i$ and $j$ from them, and for those queries containing both we add a new ball $s\not\in [n]$. The resulting $(k-1)$-uniform hypergraph $\cF'$ has less than $\mathsf{m}(k-1)$ edges, thus it has Property B. This gives a coloring on a subset of $[n]\setminus \{i,j\} \cup \{s\}$ without a monochromatic edge, let red be the color of $s$. So far less than $\mathsf{m}(k-1)(k-1)$ balls have colors, thus we can extend this coloring to all the balls except for $i$ and $j$ such a way that there are $(n-1)/2$ blue and $(n-3)/2$ red balls among the balls in $[n] \setminus \{i,j\}$. Then we answer according to this coloring, and in particular we answer YES to queries containing one of $i$ and $j$ (as both colors appear among the balls in those queries). We also answer YES to the queries containing both $i$ and $j$. We know that those queries contain a blue ball different from $i$ and $j$. Therefore, it is easy to see that if one of $i$ and $j$ is red, the answers are consistent with the coloring. Thus any color can be minority, hence another ball cannot be the majority ball, as we know its color. But $i$ can be red and $j$ blue, or the other way around. In that case the red ball is in minority, thus we cannot claim $i$ or $j$ is a majority ball, finishing the proof for the General model.

In Borzyszkowski's Model, if $n$ is even, the same proof works, as we can point out a pair of balls not equal to $i$ in every query. If $n$ is odd, the same proof does not work, as it is possible that only $s$ is of color red in a query. In that case either $i$ or $j$ has to be put in the pair that is pointed out, thus we might find out its color. However, we can show that all but one of the balls have to be contained in at least $\mathsf{m}(k-2)/2$ queries. Indeed, otherwise there are less than $\mathsf{m}(k-2)$ queries containing either $i$ or $j$. Let $\cG$ be the hypergraph that has those queries as edges, and remove $i$ and $j$ from them, and for those queries containing only one of $i$ and $j$, we remove another arbitrary ball. The resulting $(k-2)$-uniform hypergraph $\cG'$ has less than $\mathsf{m}(k-2)$ edges, thus it has Property B. From here it works as in the General model: We find a coloring without a monochromatic edge, then extend this coloring to all the balls but $i$ and $j$, such a way that there are  $(n-1)/2$ blue and $(n-3)/2$ red balls. Then answer according to this coloring, in particular we can answer YES and point out a pair of two balls of different colors without using $i$ and $j$. By the same reasoning as in the case of the General model, any ball can be a minority ball, finishing the proof.
\end{proof}

\section{Acknowledgement}

We thank the reviewer for the careful review, which helped to improve the presentation of the manuscript a lot.

Research was supported by the National Research, Development and Innovation Office -
NKFIH under the grants FK 132060, KH130371, KKP-133819 and SNN 129364. Research
of Vizer was supported by the J\'anos Bolyai Research Fellowship and by the New National
Excellence Program under the grant number UNKP-20-5-BME-45.

\end{document}